%% file: MacroscopicScalarCurvature.tex
\documentclass[11pt,reqno]{amsart}

\usepackage{geometry}
\usepackage{amsmath}
\usepackage{amssymb}

\usepackage{graphicx}
\usepackage{tikz-cd}
\graphicspath{ {./Figures/} }

\usepackage{import}
\usepackage{xifthen}
\usepackage{pdfpages}
\usepackage{transparent}

\usepackage{hyperref}
\usepackage{xcolor}
\definecolor{primarycolor}{RGB}{113,148,112}
\definecolor{secondarycolor}{RGB}{166,44,55}
\definecolor{tertiarycolor}{RGB}{193,196,148}
\definecolor{coral}{RGB}{248,131, 121}
\hypersetup{
	colorlinks,
	linkcolor = primarycolor,
	citecolor = secondarycolor,
	urlcolor = primarycolor,
	linktoc = page
}

\newcommand{\GG}[1]{}

\newtheorem{theorem}{Theorem}[section]
\newtheorem{lemma}[theorem]{Lemma}
\newtheorem{proposition}[theorem]{Proposition}
\newtheorem{corollary}[theorem]{Corollary}

\theoremstyle{definition}
	\newtheorem{definition}[theorem]{Definition}
	\newtheorem{remark}[theorem]{Remark}
	
	\newtheorem{notation}[theorem]{Notation}
	\newtheorem*{acknowledgements}{Acknowledgements}

\newcommand{\set}[1]{\left\{ #1 \right\}}
\newcommand{\abs}[1]{\left\lvert #1 \right\rvert}

\newcommand{\Oslash}{\text{\O}}
\newcommand{\N}{\mathbb{N}}
\newcommand{\Z}{\mathbb{Z}}
\newcommand{\R}{\mathbb{R}}
\newcommand{\C}{\mathbb{C}}
\newcommand{\Sp}{\mathbb{S}}
\newcommand{\T}{\mathbb{T}}
\newcommand{\RP}{\mathbb{R}\mathrm{P}}

\DeclareMathOperator{\diam}{diam}
\DeclareMathOperator{\scal}{scal}
\DeclareMathOperator{\sect}{sect}
\DeclareMathOperator{\mscal}{mscal}
\DeclareMathOperator{\UW}{UW}
\DeclareMathOperator{\Hom}{Hom}
\DeclareMathOperator{\sys}{sys}

\title[Urysohn width of hypersurfaces and positive macroscopic scalar curvature]{Urysohn width of hypersurfaces and positive macroscopic scalar curvature}

\author[T. Gil Moreno de Mora Sard\`a]{Teo Gil Moreno de Mora Sard\`a}
\address{Teo Gil Moreno de Mora Sard\`a, Univ Paris Est Cr\'eteil, CNRS, LAMA, F-94010 Cr\'eteil, France;
Univ Gustave Eiffel, LAMA, F-77447 Marne-la-Vall\'ee, France; Departament de Matem\`atiques, Universitat Aut\`onoma de Barcelona, Barcelona, Spain}
\email{teo.gil-moreno-de-mora-i-sarda@u-pec.fr}

\date{\today}

\makeatletter 
\@namedef{subjclassname@2020}{\textup{2020} Mathematics Subject Classification} 
\makeatother

\thanks{The author acknowledges support by the project Min-Max (ANR-19-CE40-0014), the FEDER/AEI/MICINN grant PID2021-125625NB-I00 and the AGAUR grant 2021-SGR-01015}

\begin{document}

\begin{abstract}
	We prove that if a complete Riemannian $n$-manifold with non-trivial codimension~1 homology with $\Z_2$-coefficients or $\Z$-coefficients has positive macroscopic scalar curvature large enough, then it contains a non-nullhomologous hypersurface of small Urysohn $(n-2)$-width. This constitutes a macroscopic analogue of a theorem by Bray--Brendle--Neves on the area of non-contractible 2-spheres in a closed Riemannian 3-manifold with positive scalar curvature. Our proof is based on an adaptation of Guth's macroscopic version of the Schoen--Yau descent argument.
\end{abstract}

\maketitle

\section{Introduction}

The scalar curvature of a Riemannian $n$-manifold $M$ is a fundamental invariant in Riemannian geometry. The scalar curvature $\scal(x)$ at a point $x \in M$ is defined as
\begin{equation*}
	\scal(x) = \sum_{i \neq j} \sect_x(e_i,e_j),
\end{equation*}
where $\sect_x$ denotes the sectional curvature of the manifold $M$ at the point $x$ and $(e_i)$ is an orthonormal basis of the tangent space $T_xM$. The scalar curvature can be equivalently defined through the volumetric deviation of geodesic balls of infinitesimal radius with respect to Euclidean balls of the same radius. More precisely, the volume of the geodesic ball $B(x,r)$ centered at a point~$x \in M$ satisfies
\begin{equation*}
	\abs{B(x,r)} = b_n r^n \left( 1 - \frac{\scal(x)}{6(n+2)} r^2 + O(r^3)\right)
\end{equation*}
for radii $r > 0$ small enough, where $\scal(x)$ denotes the scalar curvature of $M$ at the point $x$ and~$b_n$ is the volume of the unit ball in the Euclidean $n$-dimensional space.

A central problem in Riemannian geometry consists in understanding the relation between scalar curvature and the global topology and geometry of a manifold. In \cite{Bray_Brendle_Neves_2010}, the authors investigated the effect of a lower bound on the scalar curvature of a Riemannian 3-manifold on its 2-systole.

\begin{definition} \label{def:homotopicalSystole}
	Let $M$ be a Riemannian $n$-manifold with $\pi_k(M) \neq 0$ for some $k \in \set{1, \dots, n-1}$. The \emph{homotopical $k$-systole} of $M$ is defined as
	\begin{equation*}
		\sys\pi_k(M) := \inf\set{ \abs{\Sigma} \mid \Sigma \subset M \text{ immersed } k\text{-sphere such that } [\Sigma] \neq 0 \in \pi_k(M) },
	\end{equation*}
	where $\abs{\Sigma}$ denotes the $k$-dimensional volume of the $k$-sphere $\Sigma$.
\end{definition}

\begin{theorem}[\cite{Bray_Brendle_Neves_2010}] \label{thm:BrayBrendleNeves}
	Let $M$ be a closed Riemannian 3-manifold with $\pi_2(M) \neq 0$. Suppose that $\scal \geq s > 0$. Then
	\begin{equation} \label{eqn:BBNInequality}
		\sys\pi_2(M) \leq \frac{8\pi}{s}.
	\end{equation}
	Moreover, equality holds if and only if the universal cover of $M$ is isometric to the standard Riemannian cylinder~$\Sp^2(1) \times \R$ up to scaling.
\end{theorem}

The proof of Theorem \ref{thm:BrayBrendleNeves} relies on the stability formula for a non-contractible 2-sphere of least area. Theorem \ref{thm:BrayBrendleNeves} has been generalised in multiple directions. For example, Bray--Brendle--Eichmair--Neves proved an analogous inequality for embedded projective planes, see \cite{Bray_Brendle_Eichmair_Neves_2010}. In higher dimensions, one cannot expect in general a control of the 2-systole solely from a lower bound on the scalar curvature. For instance, for $n \geq 5$, consider the Riemannian product $\Sp^2(1) \times \Sp^{n-2}(r)$ of the unit round 2-sphere with the round $(n-2)$-sphere of radius $r$, which has 2-systole equal to~$4\pi$ and arbitrarily large scalar curvature when one takes $r \to 0$. However, some generalisations have been derived under further topological assumptions on the manifold $M$. For instance, Zhu proved that inequality \eqref{eqn:BBNInequality} holds up to dimension 7 if the manifold admits a non-zero degree map to $\Sp^2 \times \T^{n-2}$, see \cite{Zhu_2020}. Zhu also generalised Theorem \ref{thm:BrayBrendleNeves} to the non-compact case under suitable topological assumptions, again up to dimension 7, see \cite{Zhu_2023}. In another direction, Richard obtained an estimate for the homotopical 2-systole of $\Sp^2 \times \Sp^2$ endowed with a metric of positive scalar curvature satisfying a certain stretching condition, see \cite{Richard_2020}.

\medskip

Theorem \ref{thm:BrayBrendleNeves} has also motivated analogous results for hypersurfaces which are minimising within their homology class.

\begin{definition}
	Let $M$ be a Riemannian $n$-manifold with $H_k(M;\Z) \neq 0$ for some $k \in \set{1, \dots, n-1}$. The \emph{homological $k$-systole} of $M$ is defined as
	\begin{equation*}
		\sys H_k (M) := \inf \set{\abs{\Sigma} \mid \Sigma \subset M \text{ immersed $k$-submanifold such that } [\Sigma] \neq 0 \in H_k(M;\Z)},
	\end{equation*}
	where $\abs{\Sigma}$ denotes the $k$-dimensional volume of the submanifold $\Sigma$ in $M$.
\end{definition}

In \cite{Stern_2022}, Stern gave a direct proof of the homological analogue of Theorem \ref{thm:BrayBrendleNeves}. A generalisation to dimensions from 4 to 7 was addressed by Chu--Lee--Zhu in \cite{Chu_Lee_Zhu_2024}, where they proved an upper bound on the codimension 1 systole under a stronger curvature positivity condition, namely positive bi-Ricci curvature, and obtained a rigidity statement for the equality case.

\medskip

In \cite{Guth_2010_Metaphors}, Guth introduced a macroscopic analogue of scalar curvature, which quantifies the volumetric deviation of geodesic balls of a fixed radius. Denote by $V^n_s(R)$ the volume of any ball of radius $R$ in the simply connected $n$-dimensional space form of constant scalar curvarture $s$.

\begin{definition}
	The \emph{macroscopic scalar curvature} $\mscal(x,R)$ of a Riemannian $n$-manifold $M$ at a point $x \in M$ and scale $R > 0$ is the unique $s \in \R$ such that
	\begin{equation*}
		\abs{B_{\tilde{M}}(\tilde{x},R)} = V^n_s(R),
	\end{equation*}
	where $\tilde{x}$ is a lift of $x$ to the universal Riemannian cover $\tilde{M}$ of $M$. Equivalently, the macroscopic scalar curvature at a point $x \in M$ satisfies $\mscal(x,R) \geq s$ if and only if
	\begin{equation*}
		\abs{B_{\tilde{M}}(\tilde{x},R)} \leq V^n_s(R).
	\end{equation*}
\end{definition}

The macroscopic scalar curvature is defined through the volumes of balls in the universal cover~$\tilde{M}$ of $M$ in order to ensure that flat manifolds have macroscopic scalar curvature equal to zero at any scale.

One may wonder whether there is a macroscopic analogue of Theorem \ref{thm:BrayBrendleNeves}. The following proposition shows that one cannot hope for a control of the homotopical and homological systoles of a closed Riemannian manifold solely from a lower bound on its macroscopic scalar curvature, see~Section \ref{sec:prolateMetrics}.

\begin{proposition} \label{pro:controlSystoles}
	Let $n \geq 3$ and $k \in \set{2, \dots, n-1}$. For every $s > 0$ and every scale $R > 0$, there is a family of product Riemannian metrics $(g_a)_{a \in [\varepsilon,+\infty)}$ on $\Sp^k \times \Sp^{n-k}$ such that the following holds.
	\begin{enumerate}
		\item For any point $x \in \Sp^k \times \Sp^{n-k}$, one has $\mscal_{(\Sp^k \times \Sp^{n-k}, g_a)}(x,R) \geq s$, for every $a \in [\varepsilon,+\infty)$.
		\item The homotopical $k$-systole and the homological $k$-systole satisfy
		\begin{equation*}
			\lim_{a \to +\infty} \sys \pi_k (\Sp^k \times \Sp^{n-k}, g_a) = \lim_{a \to +\infty} \sys H_k (\Sp^k \times \Sp^{n-k}, g_a) = +\infty.
		\end{equation*}
	\end{enumerate}
\end{proposition}

However, one could hope to have an analogue of Theorem \ref{thm:BrayBrendleNeves} holding for a weaker metric invariant describing the size of topologically non-trivial hypersurfaces, as for instance their codimension 1 Urysohn width.

\begin{definition}
	Let $X$ be a metric space and $k \in \N$. The $k$-dimensional \emph{Urysohn width} $\UW_k (X)$ of $X$ is defined as the infimal positive real number $w > 0$ such that there exists a continuous map~$f: X \rightarrow Y$ into a $k$-dimensional simplicial complex $Y$ whose fibres satisfy
	\begin{equation*}
		\diam_X{(f^{-1}(y))} \leq w
	\end{equation*}
	for every $y \in Y$.
\end{definition}

Intuitively, the Urysohn $k$-width measures how close is the metric space $X$ from being $k$-dimensional. Guth \cite{Guth_2017} proved the following result, which was conjectured by Gromov in~\cite{Gromov_1986}.

\begin{theorem}[\cite{Guth_2017}] \label{thm:Guth2017}
	There exists a dimensional constant $c_n > 0$ such that the following holds. Let $M$ be a complete Riemannian $n$-manifold. Suppose that there is a radius $R > 0$ such that, for every $x \in M$, the closed geodesic ball $B(x,R)$ centered at $x$ has volume $\abs{B(x,R)} \leq c_n R^n$. Then
	\begin{equation*}
		\UW_{n-1}(M) \leq R.
	\end{equation*}
\end{theorem}

Theorem \ref{thm:Guth2017} was proven in the more general setting of metric spaces and for the Hausdorff content in \cite{LLNR_2022}. Also recently, a shorter and simpler proof of Theorem \ref{thm:Guth2017} was given by Papasoglu in \cite{Papasoglu_2020}. Nabutovsky \cite{Nabutovsky_2022} showed that the dimensional constant $c_n$ in Theorem \ref{thm:Guth2017} can be taken to be:
\begin{equation} \label{eqn:Nabutovsky_inequality}
		c_n = \frac{1}{2^{n-1} \, n!}.
	\end{equation}
	As a corollary of Theorem \ref{thm:Guth2017}, the Urysohn $(n-1)$-width of a closed Riemannian~$n$-manifold $M$ can be estimated in terms of its volume.

\begin{corollary}[\cite{Guth_2017}]
	Let $M$ be a closed Riemannian $n$-manifold. Then
	\begin{equation*}
		\UW_{n-1}(M) \leq c_n^{-1/n} \abs{M}^{1/n}.
	\end{equation*}
\end{corollary}

As a consequence, the infimum of the Urysohn $(n-2)$-width among all non nullhomologous hypersurfaces immersed in $M$ is a weaker invariant than its homological $(n-1)$-systole.

\medskip

The main result of this paper is the following macroscopic version of Theorem \ref{thm:BrayBrendleNeves}. Let $G = \Z_2$ or $\Z$. Consider a non-simply connected complete Riemannian $n$-manifold $M$ such that $H_{n-1}(M;G) \neq 0$. Notice that when the manifold $M$ is compact and $G$-orientable, having non-trivial codimension~1 $G$-homology already implies that $M$ is not simply connected, by Poincaré's Duality and the Universal Coefficient Theorem. However, it is no longer true when one considers non-compact manifolds. Consider the homotopical 1-systole $\sys\pi_1(M)$, that is, the length of the shortest non-contractible closed curve on $M$ (see Definition~\ref{def:homotopicalSystole}). Notice that if $M$ is non-compact, one may have $\sys\pi_1(M) = 0$.

\begin{theorem} \label{thm:main}
	There is a dimensional constant $\kappa_n > 0$ such that the following holds. Let~\mbox{$G = \Z_2$} or~$\Z$. Let $M$ be a non-simply connected complete Riemannian $n$-manifold such that~$H_{n-1}(M;G) \neq 0$ and $\sys\pi_1(M) > 0$. Fix $R > 0$ and $s > 0$ such that~$\kappa_n/\sqrt{s} < R < \frac{1}{2} \sys\pi_1(M)$. Suppose that~${\mscal (x,R) \geq s}$ for every point $x \in M$. Then there exists a closed embedded hypersurface $\Sigma$ such that $[\Sigma] \neq 0 \in H_{n-1}(M;G)$ and
	\begin{equation*}
		\UW_{n-2}(\Sigma) \leq \frac{n-1}{n}R.
	\end{equation*}
\end{theorem}

For instance, Theorem \ref{thm:main} holds for the Riemannian product $\Sp^1(a) \times \Sp^{n-1}(b)$ of a circle of large radius $a$ with a round $(n-1)$-sphere of sufficiently small radius $b$, for $n \geq 3$. The proof of Theorem \ref{thm:main} consists of two steps. First, we use the technique introduced by Guth~\cite{Guth_2010_Systolic} and extended by Alpert \cite{Alpert_2022} to estimate the volume of metric balls of an almost minimising hypersurface $\Sigma$. Second, we apply Theorem \ref{thm:Guth2017} to deduce an upper bound for the Urysohn $(n-2)$-width of $\Sigma$. Using Nabutovsky's value for the dimensional constant $c_n$ in \eqref{eqn:Nabutovsky_inequality}, we derive an explicit expression for the constant~$\kappa_n$, see Remark \ref{pro:kappa_n}. Unfortunately, for this value of the constant $\kappa_n$, Theorem \ref{thm:main} cannot be applied to essential manifolds, see Remark \ref{rem:essential}.

\medskip

One cannot expect Theorem \ref{thm:main} to hold for arbitrarily large values of the scale $R$, as the volume growth of the metric balls centered at a fixed point is significantly affected beyond the injectivity radius at that point, as shown by the following proposition, see Section \ref{sec:BergerMetrics}.

\begin{proposition} \label{pro:RP3}
	Fix $\kappa > 0$. There is a family of Riemannian metrics $(\bar{g}_s )_{s > 0}$ on the real projective space~$\RP^3$ satisfying the following properties.
	\begin{enumerate}
		\item \label{ite:mscal} For every $s > 0$ large enough, there is a scale $R_s > 0$ satisfying $R_s \geq \frac{1}{2}\sys\pi_1(\RP^3,\bar{g}_s)$ and ${R_s > \kappa/\sqrt{s}}$, such that
		\begin{equation*}
			\mscal_{\bar{g}_s}(x,R_s) \geq s.
		\end{equation*}
		
		\item \label{ite:UW} For every $s > 0$ large enough, every closed embedded surface $\Sigma$ in $(\RP^3,\bar{g}_s)$ such that $[\Sigma] \neq 0 \in H_2(\RP^3;\Z_2)$ has
		\begin{equation*}
			\UW_1(\Sigma) > w,
		\end{equation*}
		 for some constant $w > 0$ (which does not depend on $s$).
	\end{enumerate}
\end{proposition}

\medskip

We refer the reader to \cite{Alpert_2022,Alpert_Balitskiy_Guth_2024,Alpert_Kei_2017,Guth_2010_Metaphors,Guth_2010_Systolic,Papasoglu_2020,Sabourau_2022} for related works on macroscopic scalar curvature, Urysohn width and volume.

\medskip

The paper is structured as follows. In Section \ref{sec:volumeFunction}, we detail some properties of the function $V^n_s(R)$ which gives the volume of a ball of radius $R$ in the simply connected $n$-dimensional space form of constant scalar curvature $s$. In Section \ref{sec:stability}, we derive an estimate for the volume of balls in almost minimising hypersurfaces. In Section \ref{sec:main_thm} we prove Theorem \ref{thm:main} and we obtain an explicit expression for the dimensional constant $\kappa_n$. Finally, in Section \ref{sec:prolateMetrics} we prove Proposition \ref{pro:controlSystoles} and in Section \ref{sec:BergerMetrics} we detail the construction of Proposition \ref{pro:RP3}.

\begin{notation}
We will denote the closed metric ball in $M$ centered at the point $x$ and of radius~$R > 0$ by $B(x,R)$, and its boundary by $S(x,R) = \partial B(x,R)$. When working in other metric spaces, such as submanifolds with the induced metric or the universal Riemannian cover, we will explicit the metric space as a subindex to avoid confusion. Given a $k$-dimensional submanifold $\Sigma$ in $M$, we will denote its $k$-dimensional volume by $\abs{\Sigma}$. We will denote by $V^n_s(R)$ the volume of any ball of radius $R$ in the simply connected $n$-dimensional space form of constant scalar curvature $s$.  We will denote by
	\begin{equation*}
		b_n = \frac{\pi^{n/2}}{\Gamma(\frac{n}{2}+1)}
	\end{equation*}
the $n$-dimensional volume of the unit ball of $\R^n$, where $\Gamma$ denotes the Gamma function, and by $w_n = (n+1) b_{n+1}$ the $n$-dimensional volume of the unit $n$-sphere.
\end{notation}

\begin{acknowledgements}
	The author is grateful to Florent Balacheff and Stéphane Sabourau for their help and patience in discussing the details of this paper and their reading of the preliminary version. The author would like to thank the referees for their detailed and useful comments.
\end{acknowledgements}

\section{The function $V^n_s(R)$} \label{sec:volumeFunction}

Let $\mathbb{M}^n_\sigma$ denote the simply connected $n$-dimensional space form of constant sectional curvature~$\sigma$. By the Killing-Hopf Theorem \cite[Theorem 3.82]{GHL_2004}, if $\sigma > 0$, $\sigma = 0$ or $\sigma < 0$ then $\mathbb{M}^n_\sigma$ is, up to rescaling of the metric, isometric to the round $n$-sphere, the Euclidean $n$-space or the hyperbolic $n$-space, respectively. The space form $\mathbb{M}^n_\sigma$ has constant scalar curvature equal to $s = n(n-1)\sigma$. Define the radius $\rho$ of~$\mathbb{M}^n_\sigma$ to be $\rho = 1/\sqrt{\abs{\sigma}}$. Recall that $V^n_s(R)$ denotes the volume of any ball of radius $R$ in the space form~$\mathbb{M}^n_\sigma$ of constant sectional curvature $\sigma = s/n(n-1)$, and $w_k$ denotes the $k$-dimensional volume of the unit round $k$-sphere. The quantity $V^n_s(R)$ can be expressed explicitly in terms of $s \in \R$, the radius~$R > 0$ and the dimension $n$.

\begin{lemma}[{\cite[Section 3.H.3]{GHL_2004}}] \label{lem:volumeFunction}
	Let $s \in \R$ and $R > 0$. Then,
	\begin{equation*}
		V^n_s(R) =\begin{cases}
					w_{n-1} {\displaystyle \int_0^R \left(\frac{\sin{(\sqrt{\sigma}t)}}{\sqrt{\sigma}}\right)^{n-1}} dt, & \text{if }  s > 0 \text{ and } R < \pi \rho \\
					w_n \rho^n, & \text{if } s > 0 \text{ and } R \geq \pi \rho \\
					b_n R^n, & \text{if } s = 0 \\
					w_{n-1} {\displaystyle \int_0^R \left(\frac{\sinh{(\sqrt{-\sigma}t)}}{\sqrt{-\sigma}}\right)^{n-1}} dt, & \text{if }  s < 0 \\
				\end{cases}.
\end{equation*}
\end{lemma}

From the explicit form of Lemma \ref{lem:volumeFunction}, we obtain the following corollary.

\begin{corollary} \label{cor:propertiesVolumeFunction}
	The function $V^n_s(R)$ satisfies the following properties.
	\begin{enumerate}
		\item \label{ite:rescaling} Let $R > 0$ and $s \in \R$. Then, for any $\lambda > 0$,
	\begin{equation*}
		V^n_s\left(\frac{R}{\lambda}\right) = \frac{1}{\lambda^n} V^n_{s/\lambda^2}(R).
	\end{equation*}
	In particular, $V^n_s(R) = V^n_{sR^2}(1) R^n$.
		\item \label{ite:strictlyDecreasing} At every fixed scale $R > 0$, $s \mapsto V^n_s(R)$ is a strictly decreasing function (see Figure \ref{fig:V_vs_scal}), which satisfies
	\begin{equation*}
		\lim_{s \to -\infty} V^n_s(R) = +\infty \text{ and } \lim_{s \to +\infty} V^n_s(R) = 0.
	\end{equation*}
	\end{enumerate}
\end{corollary}

\begin{figure}[ht]
	\includegraphics[width=0.75\textwidth]{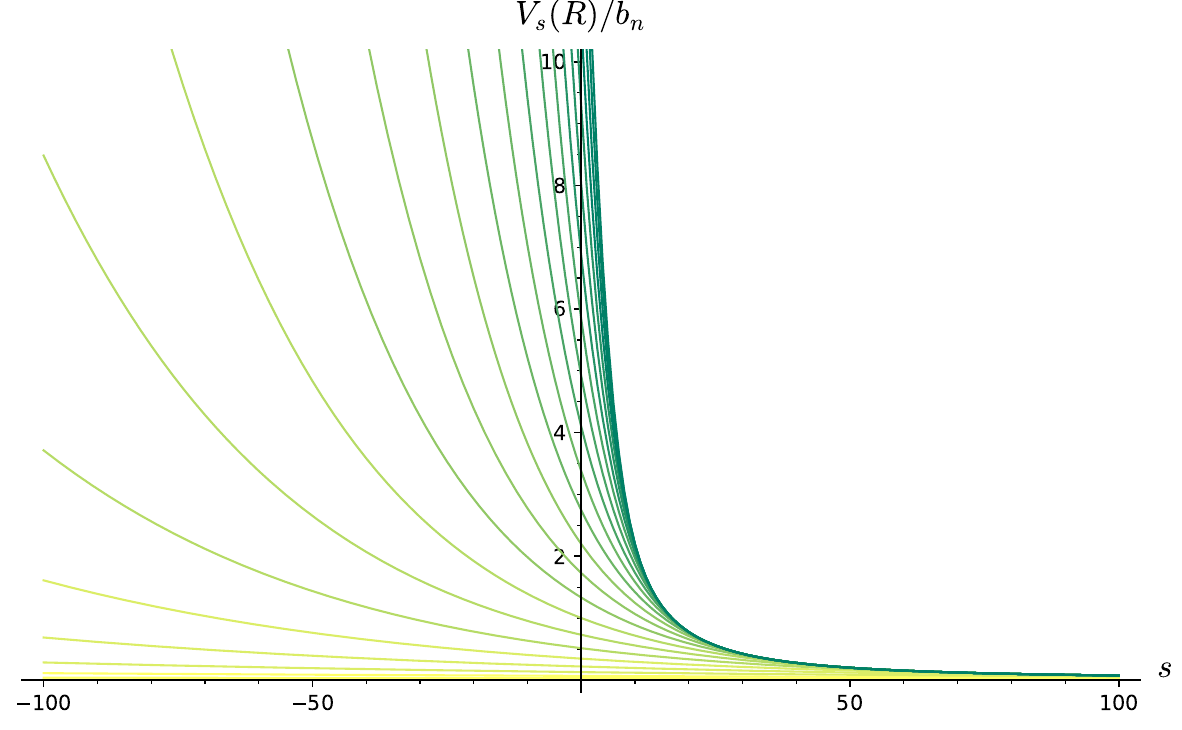}
	\caption{The function $s \mapsto V^n_s(R)/b_n$ for $n=3$ and $R = 0.1 , 0.2, \dots , 2.5$.}
	\label{fig:V_vs_scal}
\end{figure}

\section{The Stability Lemma} \label{sec:stability}

The first part of the proof of Theorem \ref{thm:main} is based on the fact that hypersurfaces that are almost minimising satisfy a convenient stability inequality, that we will prove in this section. We start by rigorously definining the notion of almost minimising hypersurface.

\begin{definition} \label{def:almost_minimising} Let $G = \Z_2$ or $\Z$. Let $M$ be a complete Riemannian $n$-manifold with $H_{n-1}(M;G) \neq 0$. Let $\Sigma$ be a closed hypersurface embedded in $M$ such that $[\Sigma] \neq 0 \in H_{n-1}(M;G)$. The hypersurface $\Sigma$ is \emph{$\delta$-almost minimising in its $G$-homology class} if any embedded hypersurface $\Sigma'$ homologous to $\Sigma$ in $H_{n-1}(M;G)$ satisfies
	\begin{equation*}
		\abs{\Sigma} \leq \abs{\Sigma'} + \delta.
	\end{equation*}
\end{definition}	

A crucial step in the proof of Theorem \ref{thm:main} is the Stability Lemma, which consists of an estimate of the volume of metric balls on an almost minimising surface in terms of the volume of balls in the ambient manifold. Originally, the Stability Lemma was developed for $\Z_2$-coefficients by Guth in~\cite{Guth_2010_Systolic} in order to give a shorter proof of Gromov's Isosystolic Inequality for the $n$-torus~\cite{Gromov_1983}. Recently, the Stability Lemma it has been extended to $\Z$-coefficients by Alpert in~\cite{Alpert_2022}.

\begin{lemma}[Stability Lemma \cite{Guth_2010_Systolic,Alpert_2022}] \label{lem:stability}
	Let $G = \Z_2$ or $\Z$. Let $M$ be a non-simply connected complete Riemannian $n$-manifold such that $H_{n-1}(M;G) \neq 0$ and $\sys\pi_1(M) > 0$. Let $\Sigma$ be a hypersurface embedded in $M$ which is $\delta$-almost minimising in its $G$-homology class. Fix $r > 0$ and~$R > 0$ such that $0 < r < R < \frac{1}{2}\sys\pi_1(M)$. Then, for every $x \in \Sigma$,
	\begin{equation*}
		\abs{B_\Sigma (x,r)} \leq \frac{1}{R-r} \abs{B(x,R)} + \delta.
	\end{equation*}
\end{lemma}
We reproduce the proofs of \cite{Guth_2010_Systolic} and \cite{Alpert_2022}, giving more detail.

\begin{proof}
	Let $x \in \Sigma$ be a point. The Coarea Formula \cite[Theorem 13.4.2]{1988_Burago_Zalgaller} for the function $d(x,\cdot)$ giving the distance to the point $x$ is
	\begin{equation*}
		\abs{B(x,R)} = \int_0^R \abs{S(x,\tau)} d\tau.
	\end{equation*}
	Hence there exists a radius~$t \in (r,R)$ for which
	\begin{equation} \label{eqn:coarea}
		\abs{B(x,R)} \geq \int_r^R \abs{S(x,\tau)} d\tau = (R-r) \abs{S(x,t)}.
	\end{equation}
	Consider the closed ball $B(x,t)$ of radius $t$. We will use the following lemma due to Gromov to prove that every closed curve lying in $B(x,t)$ intersects $\Sigma$ trivially.
	
	\begin{lemma}[{Gromov's Curve Factoring Lemma \cite[Proposition 5.28]{Gromov_1981}}] \label{lem:curveFactoring}
		Let $\gamma$ be a closed curve contained in the closed geodesic ball $B(x,t)$ and let $\varepsilon > 0$. Then $\gamma$ is $\Z$-homologous to a 1-cycle $\sum_i \gamma_i$, where each $\gamma_i$ is a closed curve of length $\ell{(\gamma_i)} < 2t + \varepsilon$.
	\end{lemma}
	
	Suppose that $\gamma$ is a closed curve lying in $B(x,t)$ with non-trivial intersection with $\Sigma$. Let $\varepsilon = \sys\pi_1(M) - 2R$. By Lemma \ref{lem:curveFactoring}, the closed curve $\gamma$ is $\Z$-homologous to a sum $\sum_i \gamma_i$ of closed curves $\gamma_i$ of length $\ell(\gamma_i) < 2t + \varepsilon$. Since $\gamma$ intersects non-trivially the hypersurface $\Sigma$, one curve~$\gamma_j$ in the sum $\sum_i \gamma_i$ has non-trivial intersection with $\Sigma$. In particular, the curve $\gamma_j$ has to be non contractible. However,
	\begin{equation*}
		\ell(\gamma_j) < 2t + \varepsilon \leq \sys\pi_1(M),
	\end{equation*}
	which is a contradiction.
	
	Consider the cycle $\Sigma \cap B(x,t)$ in $B(x,t)$ relative to the boundary $S(x,t)$. Notice that, since $t < \frac{1}{2} \sys\pi_1(M)$, the ball $B(x,t)$ is orientable. Otherwise the ball $B(x,t)$ would contain a closed curve along which the orientation of $B(x,t)$ (and of the manifold $M$) is reversed. Such a curve is non-contractible in $M$. Then, by Gromov's Curve Factoring Lemma \ref{lem:curveFactoring}, there would exist a non-contractible closed curve $\gamma_i$ of length $\ell(\gamma_i) < \sys\pi_1(M)$, which is a contradiction. Hence, by Lefschetz's Duality \cite[Theorem 3.43]{Hatcher_2002} and the Universal Coefficient Theorem \cite[Theorem 3.2]{Hatcher_2002}, there are isomorphisms
	\begin{equation*}
		H_{n-1}(B(x,t),S(x,t); \Z) \simeq H^1(B(x,t);\Z) \simeq \Hom (H_1(B(x,t);\Z), \Z).
	\end{equation*}
	Since every 1-cycle in $B(x,t)$ (and, in particular, those defining a non-torsion class in $H_1(B(x,t);\Z)$) intersects $\Sigma$ trivially, we have
	\begin{equation*}
		[\Sigma \cap B(x,t)] = 0 \in H_{n-1}(B(x,t),S(x,t);\Z),
	\end{equation*}
	which implies that the chain $\Sigma \cap B(x,t)$ is $\Z$-homologous to a chain $\sum_i m_i Z_i$ in $S(x,t)$, where $Z_i \subset S(x,t)$ are connected components of $S(x,t) \setminus \Sigma$ and $m_i \in \Z$. We will discuss the cases $G = \Z_2$ and $G=\Z$ separately hereafter.
	\medskip
	
	Let us first discuss the case $G = \Z_2$. Projecting to the chain complex with $\Z_2$-coefficients, we obtain that the chain $\Sigma \cap B(x,t)$ is $\Z_2$-homologous to a chain $\sum_i Z_i$. Consider the embedded hypersurface $\Sigma'$ obtained from $\Sigma$ by replacing $\Sigma \cap B(x,t)$ with~$\cup_i Z_i \subset S(x,t)$ and smoothing out the resulting cycle. Since the hypersurface $\Sigma$ is $\Z_2$-homologous to $\Sigma'$, the $\delta$-almost minimality of~$\Sigma$ implies
	\begin{equation*}
		\abs{\Sigma \cap B(x,t)} \leq \abs{\cup_i Z_i} + \delta \leq \abs{S(x,t)} + \delta.
	\end{equation*}
	We conclude by noting that $B_\Sigma(x,r) \subset \Sigma \cap B(x,t)$ and using the inequality \eqref{eqn:coarea}.
	
	\medskip	
	
	Finally, we address the case $G = \Z$. For $\Z$-coefficients, the chain $\Sigma \cap B(x,t)$ may fail to be~$\Z$-homologous to a chain~$\sum_i Z_i$ in~$S(x,t)$ consisting of a disjoint union of connected components $Z_i$ of~$S(x,t) \setminus \Sigma$. Still, the different connected components of $\Sigma \cap B(x,t)$ may be grouped into a collection $D_1, \dots, D_{N-1}$ such that~$\abs{D_i} \leq \abs{S(x,t)} + \delta$ for every $i \in \set{1, \dots, N-1}$.
	
	We proceed as follows. Since every closed curve lying in $B(x,t)$ has trivial intersection with $\Sigma$, one can group the connected components of $B(x,t) \setminus \Sigma$ into levels $L_1, \dots, L_N$ in a way such that every path starting at~$L_i$ and ending at $L_j$ has signed intersection number with $\Sigma$ equal to~$j-i$. For each $i \in \set{1, \dots, N}$, define~$S_i := L_i \cap S(x,t)$. Finally, group the connected components of~$\Sigma \cap B(x,t)$ into dividers~$D_1, \dots, D_{N-1}$ so that the divider $D_i$ is the common boundary between $L_i$ and~$L_{i+1}$ for~$i \in \set{1,\dots, N-1}$, see Figure \ref{fig:ball}. For convenience, we set $D_0 = \Oslash$ and $D_N = \Oslash$.
	
	\begin{figure}[ht]
		\centering
    \def\svgwidth{0.65\columnwidth}
    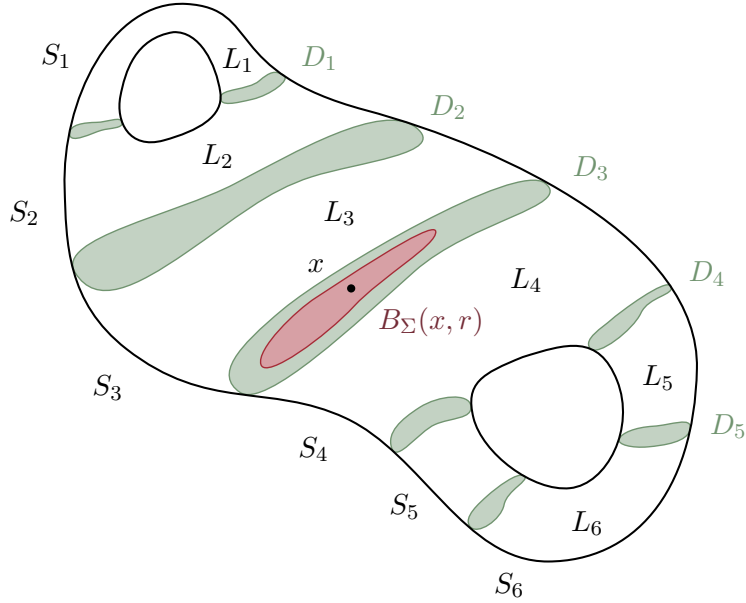

		\caption{Subdivision of the ball $B(x,t)$ into levels $L_1, \dots, L_N$ separated by the dividers $D_1, \dots, D_N$.}
		\label{fig:ball}
	\end{figure}
	
	For every $i \in \set{1, \dots, N-1}$ and every $k \in \set{0, \dots, N}$, consider the chain $D_{i,k}$ defined by
	\begin{equation*}
		D_{i,k} :=\begin{cases}
					D_k + \sum_{j=k+1}^i S_j , & \text{if }  k < i \\
					D_i, & \text{if } k = i \\
					D_k + \sum_{j = i+1}^k S_j, & \text{if }  k > i \\
				\end{cases}.
	\end{equation*}
	In particular, we have $D_{i,0} = \sum_{j = 1}^i S_j$ and $D_{i,N} = \sum_{j=i+1}^N S_j$ the two connected components of~$S(x,t) \setminus D_i$. Notice that, for every $i \in \set{1, \dots, N-1}$ and every $k \in \set{0, \dots, N}$, the chain $D_i$ is~$\Z$-homologous to $D_{i,k}$. For each $i \in \set{1, \dots, N-1}$, let $k_i \in \set{0, \dots, N}$ be such that
	\begin{equation*}
		\abs{D_{i,k_i}} = \min_{k \in \set{0, \dots, N}} \abs{D_{i,k}},
	\end{equation*}
	and define $D'_i := D_{k_i}$. That is, for each $i \in \set{1, \dots, N-1}$, the chain $D'_i$ denotes the combination~$D_{i,k}$ of least area. By the minimality of the $D_{k_i}$ with respect to the combinations $D_{i,k}$, one can always assume that $0 \leq k_1 \leq \dots \leq k_{N-1} \leq N$.
	
	Now, modify the hypersurface $\Sigma$ by replacing each $D_i$ by the corresponding $D'_i$ and perturb the resulting hypersurface to make it embedded, see Figure \ref{fig:deformation}.
	
	\begin{figure}[ht]
		\centering
    \def\svgwidth{1\columnwidth}
    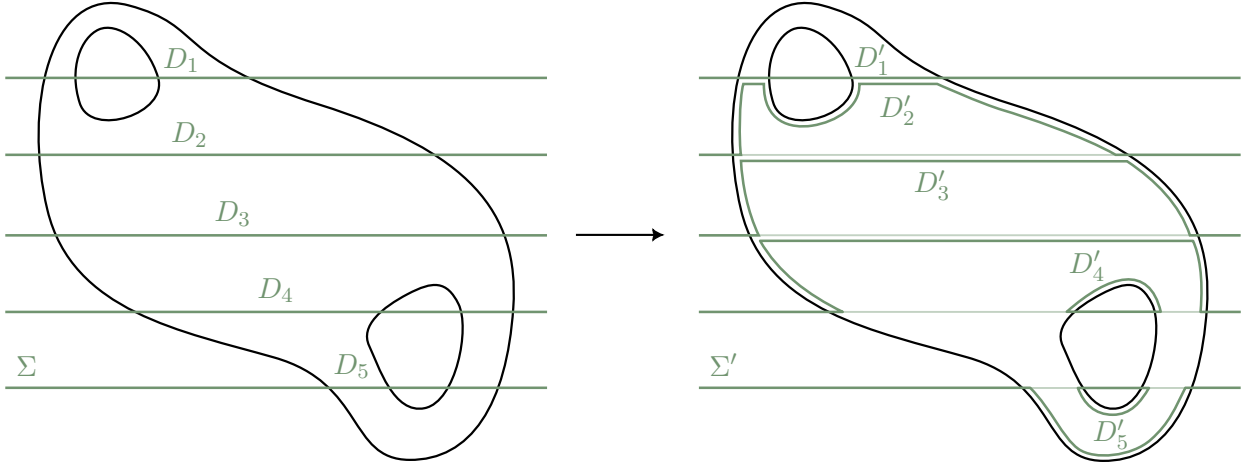

		\caption{Modification of the surface $\Sigma$ by replacing each $D_i$ by the corresponding $D'_i$.}
		\label{fig:deformation}
	\end{figure}
	
 We obtain an embedded hypersurface $\Sigma'$ which is $\Z$-homologous to the original hypersurface $\Sigma$. By the $\delta$-almost minimality of $\Sigma$, we have
	\begin{equation}
	 \label{ine:almost_minimality}
		\sum_{i = 1}^{N-1} \abs{D_i} \leq \sum_{i = 1}^{N-1} \abs{D'_i} + \delta.
	\end{equation}
	From the inequality \eqref{ine:almost_minimality} and the minimality of the $D'_i$, it follows that, for every $i \in \set{1, \dots, N-1}$,
	\begin{equation*}
		\abs{D_i} \leq \abs{D'_i} + \delta.
	\end{equation*}
	The minimality of $D'_i$ implies that $\abs{D'_i} \leq \sum_{j=0}^i \abs{S_j}$ and $\abs{D'_i} \leq \sum_{j=i+1}^N \abs{S_j}$. We derive that for every $i \in \set{1, \dots, N-1}$,
	\begin{equation*}
		\abs{D_i} \leq \abs{S(x,t)} + \delta.
	\end{equation*}
	We conclude by observing that there is an $i \in \set{1, \dots, N-1}$ such that $B_\Sigma (x,r) \subset D_i$ and using the inequality \eqref{eqn:coarea}.
\end{proof}

\section{Proof of the main theorem} \label{sec:main_thm}

Now we prove Theorem \ref{thm:main}.

\begin{proof}[Proof of Theorem \ref{thm:main}] \label{prf}
	Let $G = \Z_2$ or $\Z$. Recall that $M$ is a non-simply connected complete Riemannian $n$-manifold with $H_{n-1}(M;G) \neq 0$ and $\sys\pi_1(M) > 0$. Fix any non-trivial homology class $h \in H_{n-1}(M;G)$. Recall that every codimension 1 homology class with coefficients either in~$\Z_2$ or in $\Z$ can be represented by a smooth closed embedded hypersurface. Let $\delta > 0$ be a positive number to be chosen later. Let $\Sigma$ be a~$\delta$-almost minimising closed embedded hypersurface representing~$h$. Fix a point $x \in \Sigma$ and consider two radii $0 < r < R < \frac{1}{2}\sys\pi_1(M)$ to be determined later. Let $B_\Sigma (x,r)$ be the metric ball centered at the point $x$ of radius $r$ with respect to the induced metric on $\Sigma$. The Stability Lemma \ref{lem:stability} together with the lower bound on the macroscopic scalar curvature of $M$ at point $x$ and scale $R$ imply
	\begin{equation*} \label{eqn:main_stability}
		\abs{B_\Sigma (x,r)} \leq \frac{1}{R-r} V^n_s(R) + \delta.
	\end{equation*}
	If, given $0 < r < R$ and $s > 0$, the inequality
	\begin{equation} \label{eqn:main_estimateVolumes}
		\frac{1}{R-r} V^n_s(R) < c_{n-1} r^{n-1}
	\end{equation}
	holds, then Theorem \ref{thm:Guth2017} applied to the $\delta$-almost minimising hypersurface $\Sigma$ for $\delta > 0$ small enough implies that $\UW_{n-2}(\Sigma) \leq r$. By Corollary \ref{cor:propertiesVolumeFunction}.\eqref{ite:rescaling}, the inequality \eqref{eqn:main_estimateVolumes} is equivalent to
	\begin{equation} \label{eqn:main_estimateVolumes2}
		\frac{1}{(r/R)^{n-1}} \,\cdot \frac{1}{1-r/R} V^n_{sR^2} (1)< c_{n-1}.
	\end{equation}
	The value of the inner radius $r$ that makes the left-hand term in the inequality \eqref{eqn:main_estimateVolumes2} as small as possible is $r = \frac{n-1}{n} R$. In this case, the inequality \eqref{eqn:main_estimateVolumes2} becomes
	\begin{equation*}
		V^n_{sR^2}(1) < \frac{(n-1)^{n-1}}{n^n} c_{n-1},
	\end{equation*}
	which is equivalent to $sR^2 > \kappa_n^2$, where $\kappa_n := \sqrt{f_n \left(\frac{(n-1)^{n-1}}{n^n} c_{n-1}\right)}$ is the dimensional constant in Theorem \ref{thm:main}. Here, $f_n:(0,\infty) \rightarrow \R$ is the inverse function of the map $s \mapsto V_s^n(1)$, see Corollary \ref{cor:propertiesVolumeFunction} \eqref{ite:strictlyDecreasing}.
\end{proof}

\smallskip

\begin{remark} \label{pro:kappa_n}
	One may take the dimensional constant $\kappa_n$ in Theorem \ref{thm:main} to be
	\begin{equation*}
		\kappa_n = (n-1)^{\frac{1}{n}-\frac{1}{2}} \, n^{\frac{3}{2}} \, 2^{1-\frac{2}{n}} \, (n-1)!^{\frac{1}{n}} \, w_n^{\frac{1}{n}} = \sqrt{\frac{8\pi}{e}} \, n^{\frac{3}{2}} \, \left(1+O \left( \frac{\log{n}}{n}\right)\right).
	\end{equation*}
	Indeed, the constant $\kappa_n$ is implicitly defined by the equation
	\begin{equation} \label{eqn:def_kappa}
		V^n_{\kappa_n^2}(1) = \frac{(n-1)^{n-1}}{n^n} c_{n-1},
	\end{equation}
	where $c_{n-1}$ is the dimensional constant introduced in Theorem \ref{thm:Guth2017}. Now, let $\rho > 0$ be such that the quantity in equation \eqref{eqn:def_kappa} coincides with the volume of the round $n$-sphere of radius $\rho$, that is
	\begin{equation*}
		\frac{(n-1)^{n-1}}{n^n} c_{n-1} = w_n \rho^n = V^n_{n(n-1)/\rho^2}(\pi\rho).
	\end{equation*}
	It is straightforward to check that, for Nabutovsky's value \eqref{eqn:Nabutovsky_inequality} of the constant $c_{n-1}$, one has
	\begin{equation*}
		\pi \rho = \pi (n-1)^{1-\frac{1}{n}} \, n^{-1} \, 2^{\frac{2}{n}-1} \, (n-1)!^{-\frac{1}{n}} \, w^{-\frac{1}{n}}_n \leq 1,
	\end{equation*}
	for every $n \geq 2$. Hence, by Lemma \ref{lem:volumeFunction}, we obtain
	\begin{equation*}
		V^n_{n(n-1)/\rho^2}(\pi\rho) = V^n_{n(n-1)/\rho^2}(1).
	\end{equation*}
	Using Corollary \ref{cor:propertiesVolumeFunction} \eqref{ite:strictlyDecreasing}, it follows that
	\begin{equation} \label{eqn:kappa_estimate}
		\kappa_n = \frac{\sqrt{n(n-1)}}{\rho} = (n-1)^{\frac{1}{n}-\frac{1}{2}} \, n^{\frac{3}{2}} \, 2^{1-\frac{2}{n}} \, (n-1)!^{\frac{1}{n}} \, w^{\frac{1}{n}}_n.
	\end{equation}
	Finally, applying Stirling's formula \cite[Equation 2.6.1]{OMS_2009} to equation \eqref{eqn:kappa_estimate}, we obtain the following asymptotic behaviour
	\begin{equation*}
		\kappa_n = \sqrt{\frac{8\pi}{e}} n^{\frac{3}{2}} \left(1 + O\left(\frac{\log{n}}{n}\right)\right).
	\end{equation*}
\end{remark}

\begin{remark} \label{rem:essential}
	One may wonder whether Theorem \ref{thm:main} is applicable to essential manifolds. Recall that if $M$ is an essential Riemannian $n$-manifold, then
	\begin{equation}\label{eqn:essential_sys_UW}
		\sys\pi_1(M) \leq 2 \UW_{n-1}(M).
	\end{equation}
	Inequality \eqref{eqn:essential_sys_UW} was first established by Gromov \cite[Lemma 1.2.B, Appendix 1]{Gromov_1983}, and then improved in \cite[Equation (4)]{Nabutovsky_2022}. Now let $M$ be an essential Riemannian $n$-manifold and fix a scale $R > \kappa_n / \sqrt{s}$. Suppose that the macroscopic scalar curvature of $M$ satifies $\mscal(x,R) \geq s$ for every $x \in M$. Then, for every $x \in M$,
	\begin{equation*}
		\abs{B(x,R)} \leq \frac{(n-1)^{n-1}}{n^n} c_{n-1} R^n.
	\end{equation*}
	Using Nabutovsky's value of the dimensional constant $c_n$ in \eqref{eqn:Nabutovsky_inequality}, one has
	\begin{equation*}
		\frac{(n-1)^{n-1}}{n^n} c_{n-1} < c_n.
	\end{equation*}
	Therefore, by Theorem \ref{thm:Guth2017} and the inequality \eqref{eqn:essential_sys_UW}, it follows that $\frac{1}{2} \sys\pi_1(M) \leq R$. Hence the hypothesis of Theorem~\ref{thm:main} cannot be satisfied for essential manifolds.
\end{remark}

\section{Prolate product metrics on $\Sp^k \times \Sp^{n-k}$}
\label{sec:prolateMetrics}

In this section we will prove Proposition \ref{pro:controlSystoles}. We will construct a family of Riemannian metrics on $\Sp^k \times \Sp^{n-k}$ given by the Riemannian product of a prolate hyperellipsoidal metric on $\Sp^k$ with the standard round metric on $\Sp^{n-k}$. Given $0 < \varepsilon \leq a$, we define the \emph{$k$-dimensional prolate hyperellipsoid} to be
\begin{equation} \label{eqn:prolate}
		E^k(\varepsilon, a) = \set{\frac{x_1^2}{\varepsilon^2} + \cdots + \frac{x_k^2}{\varepsilon^2} + \frac{x_{k+1}^2}{a^2} = 1} \subset \R^{k+1}.
	\end{equation}
We shall make use of the following estimate of the volume of the metric balls in $E^k(\varepsilon, a)$. Recall that $w_k$ denotes the $k$-volume of the unit round $k$-sphere.

\begin{lemma} \label{lem:prolate}
	Let $B_{E^k(\varepsilon, a)}(x,R)$ be a metric ball in the $k$-dimensional prolate hyperellipsoid $E^k(\varepsilon, a)$. Then
	\begin{equation*}
		\abs{B_{E^k(\varepsilon, a)}(x,R)} \leq 2 w_{k-1} \varepsilon^{k-1} R.
	\end{equation*}
\end{lemma}

\begin{proof}[Proof of Lemma \ref{lem:prolate}]
		Denote by $x_0 := (0, \dots, 0, -a)$ the south pole of the prolate hyperellipsoid~$E^k(\varepsilon, a)$. Consider the function $h(y) := d_{E^k(\varepsilon, a)}(x_0,y)$ giving the intrinsic distance from the south pole $x_0$ to the point~$y \in E^k(\varepsilon, a)$. Notice that, for every $t \in \R$, the level set $h^{-1}(t)$ is either empty or a parallel of the prolate ellipsoid, and therefore isometric to a round sphere $\Sp^{k-1}(\rho(t))$ of radius $0 \leq \rho(t) \leq \varepsilon$.
		
		The metric ball $B_{E^k(\varepsilon, a)}(x,R)$ is contained in the metric annulus $h^{-1}([h_0 - R, h_0 + R])$, where $h_0 := h(x) = d(x_0,x)$ denotes the distance of the point $x$ to the south pole $x_0$, see Figure~\ref{fig:prolate}.
		\begin{figure}[ht]
		\centering
    \def\svgwidth{0.40\columnwidth}
    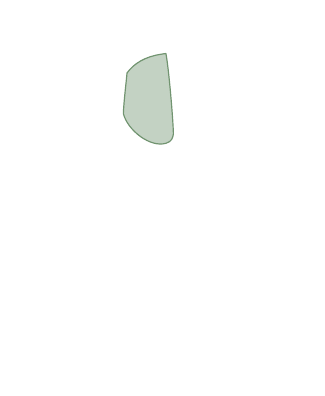

		\caption{The metric ball $B_{E^k(\varepsilon, a)}(x,R)$ is contained inside the metric annulus \mbox{$h^{-1}([h_0 - R, h_0 + R])$} of the prolate ellipsoid $E^k(\varepsilon,a)$.}
		\label{fig:prolate}
	\end{figure}
		Hence, using the Coarea Formula \cite[Theorem 13.4.2]{1988_Burago_Zalgaller}, we obtain
		\begin{equation*}
			\abs{B_{E^k(\varepsilon, a)}(x,R)} \leq \abs{h^{-1}([h_0 - R, h_0 + R])} = \int_{h_0-R}^{h_0+R} \abs{\Sp^{k-1}(\rho(t))} dt \leq 2R\abs{\Sp^{k-1}(\varepsilon)} = 2 R w_{k-1}\varepsilon^{k-1}.
		\end{equation*}
\end{proof}

Now we prove Proposition \ref{pro:controlSystoles}.

\begin{proof}[Proof of Proposition \ref{pro:controlSystoles}]
	Fix $s > 0$ and a scale $R > 0$. Let
	\begin{equation} \label{eqn:varepsilon}
		\varepsilon := \left(\frac{V^n_s(R)}{2 w_{k-1} R \, V^{n-k}_{(n-k)(n-k-1)}(R)}\right)^{\frac{1}{k-1}}.
	\end{equation}
	Notice that $\varepsilon$ depends on both $s$ and $R$. For every $a \geq \varepsilon$, consider the product Riemannian manifold
	\begin{equation*}
		(M,g_a) = E^k(\varepsilon,a) \times \Sp^{n-k}(1),
	\end{equation*}
	which is diffeomorphic to $\Sp^k \times \Sp^{n-k}$. The universal Riemannian cover $(\tilde{M},\tilde{g}_a)$ of $(M,g_a)$ is given by the Riemannian product of $E^k(\varepsilon,a)$ with the universal Riemannian cover $\tilde{\Sp}^{n-k}(1)$ of $\Sp^{n-k}(1)$. Notice that $\tilde{\Sp}^{n-k}(1)$ is isometric to the round $(n-k)$-sphere if $1 \leq k \leq n-2$, and the standard real line for $k = n-1$. 
	
	For every $a \geq \varepsilon$, the $k$-systoles of $(M,g_a)$ are given by
	\begin{equation*}
		\sys\pi_k (M,g_a) = \sys H_k (M,g_a) = \abs{E^k(\varepsilon,a)},
	\end{equation*}
	and it is easy to see that $\lim\limits_{a \to +\infty} \abs{E^k(\varepsilon,a)} = +\infty$.
	
	\medskip
	
	Now fix a point $x \in \Sp^k \times \Sp^{n-k}$. Consider the metric ball $B_{(\tilde{M},\tilde{g}_a)}(\tilde{x},R)$ of radius~$R$ centered at a lift $\tilde{x}$ of $x$ to $\tilde{M}$. Since $(\tilde{M},\tilde{g}_a)$ is a Riemannian product, the ball~$B_{(\tilde{M},\tilde{g}_a)}(\tilde{x},R)$ is contained in the product of metric balls
	\begin{equation*}
		B_{E^k(\varepsilon, a)}(\tilde{x}_1,R) \times B_{\tilde{\Sp}^{n-k}(1)}(\tilde{x}_2, R),
	\end{equation*}
	where $\tilde{x}_1$ and $\tilde{x}_2$ denote the projections of $\tilde{x}$ to the corresponding factors. By Lemma \ref{lem:prolate}, the volume of the ball $B_{E^k(\varepsilon, a)}(\tilde{x}_1,R)$ satisfies
	\begin{equation*}
		\abs{B_{E^k(\varepsilon, a)}(\tilde{x}_1,R)} \leq 2w_{k-1} \varepsilon^{k-1} R.
	\end{equation*}
	On the other hand, the quantity $\abs{B_{\tilde{\Sp}^{n-k}(1)}(\tilde{x}_2, R)}$ coincides with the volume $V^{n-k}_{(n-k)(n-k-1)}(R)$ of any ball of radius $R$ in the unit round~$(n-k)$-sphere, which has scalar curvature $(n-k)(n-k-1)$. Therefore, using equation \eqref{eqn:varepsilon}, we have
	\begin{equation*}
		\abs{B_{(\tilde{M},\tilde{g}_a)}(\tilde{x},R)}  \leq 2w_{k-1} \varepsilon^{k-1} R \, V^{n-k}_{(n-k)(n-k-1)}(R) = V^n_s(R).
	\end{equation*}
	Hence, for every $a \geq 1$, we have
	\begin{equation*}
		\mscal_{(M,g_a)}(x,R) \geq s
	\end{equation*}
	uniformly in $x \in M$.
\end{proof}

\section{Berger metrics on $\RP^3$} \label{sec:BergerMetrics}

Finally, let us construct the family of Riemannian metrics on the real projective space $\RP^3$ presented in Proposition \ref{pro:RP3}.

It will be convenient to identify the 3-sphere with $\Sp^3 = \{(z,w) \mid \abs{z}^2 + \abs{w}^2 = 1\} \subset \C^2$, and the~2-sphere with $\Sp^2 = \{(z,t) \mid \abs{z}^2 + t^2 = 1\} \subset \C \times \R$. Consider the \emph{Hopf action} on $\Sp^3$, that is, the free action of $\Sp^1$ on $\Sp^3$ given by
\begin{equation*}
	\theta \cdot (z,w) = (e^{i\theta} z, e^{i\theta} w),
\end{equation*}
for every $\theta \in \Sp^1 = \R / \Z$ and every $(z,w) \in \Sp^3$. The quotient space $\Sp^3 / \Sp^1$ corresponding to the Hopf action is homeomorphic to $\Sp^2$, and the projection $\Sp^3 \rightarrow \Sp^2$ defines a circle bundle structure on $\Sp^3$. Let $V_{(z,w)} = (iz,iw)\in \C^2$ denote the \emph{Hopf vector field}, which is a unit vector field on $\Sp^3$ (with respect to the round metric) tangent to the orbits of the Hopf action.

\medskip

Now consider the 1-parameter family of \emph{Berger metrics} on $\Sp^3$, given by
\begin{equation*}
	g_\varepsilon (X,Y) = g(X,Y) + (\varepsilon^2 - 1) g(X,V) g(V,Y), \quad \varepsilon > 0,
\end{equation*}
for any pair of vectors $X, Y$ tangent to $\Sp^3$, where $g$ denotes the standard round Riemannian metric on $\Sp^3$. Intuitively, the Berger metric $g_\varepsilon$ is obtained from the round metric $g$ by shrinking the metric in the direction of the Hopf fibres by a factor $\varepsilon$ (so that they have length $2\pi \varepsilon$ with respect to the metric $g_\varepsilon$). Notice that the Berger metric $g_\varepsilon$ corresponding to $\varepsilon = 1$ coincides with the standard round metric $g$ on~$\Sp^3$. The quotient map
\begin{align*}
    \begin{array}{lccc}
        \mathcal{H} \colon    & (\Sp^3,g_\varepsilon)    & \longrightarrow   & \Sp^2(\tfrac{1}{2}) \\
                        & (z,w) & \longmapsto       & \left(z\bar{w},\frac{1}{2}(\abs{z}^2-\abs{w}^2)\right)
    \end{array},
\end{align*}
known as the \emph{Hopf map}, is a Riemannian submersion.

\medskip

The antipodal action of $\Z_2$ on the Berger sphere $(\Sp^3,g_\varepsilon)$ is an isometric action. Hence, the real projective space $\RP^3$ inherits a Riemannian metric from $(\Sp^3,g_\varepsilon)$, that we denote by $\bar{g}_\varepsilon$. The map~$\mathcal{H}$ induces a Riemannian submersion $\bar{\mathcal{H}}: (\RP^3,\bar{g}_\varepsilon) \rightarrow \Sp^2(\frac{1}{2})$, which defines a circle bundle structure on $\RP^3$. In particular, the map $\bar{\mathcal{H}}$ is 1-Lipschitz.

\medskip

Proposition \ref{pro:RP3} follows from Proposition \ref{pro:mscal} and Proposition \ref{pro:UW}.

\begin{proposition} \label{pro:mscal}
	Fix $\kappa > 0$. For every $\varepsilon \in (0,1)$, there is a scale $R_\varepsilon > 0$ satisfying $R_\varepsilon \geq \frac{1}{2} \sys\pi_1 (\RP^3,\bar{g}_\varepsilon)$ and $R_\varepsilon > \kappa/\sqrt{s_\varepsilon}$, with $s_\varepsilon := 6/\varepsilon^{2/3}$, such that for any point~$x \in \RP^3$,
	\begin{equation*}
		\mscal_{ (\RP^3,\bar{g}_\varepsilon)}(x,R_\varepsilon) \geq s_\varepsilon.
	\end{equation*}
\end{proposition}

\begin{proof}
	Fix $\varepsilon \in (0,1)$. Let $\kappa' > \max\set{\kappa/\sqrt{6}, \pi}$ be a constant, and consider the scale $R_\varepsilon = \kappa' \sqrt[3]{\varepsilon}$. Notice that $R_\varepsilon > \kappa / \sqrt{s_\varepsilon}$ and $R_\varepsilon > \sys\pi_1(\RP^3,\bar{g}_\varepsilon) = \pi \varepsilon$. By the Coarea Formula \cite[Theorem 13.4.2]{1988_Burago_Zalgaller} applied to the fibration $\mathcal{H}: (\Sp^3,g_\varepsilon) \rightarrow \Sp^2(\tfrac{1}{2}) $
	we have
	\begin{equation*}
		\abs{B_{(\Sp^3,g_\varepsilon)}(\tilde{x}, R_\varepsilon)} \leq \abs{(\Sp^3,g_\varepsilon)}\leq 2\pi \varepsilon \abs{\Sp^2(\tfrac{1}{2})} = 2\pi^2 \varepsilon.
	\end{equation*}
	Notice that the volume of the unit 3-sphere is $2\pi^2$, that is, $w_3 = 2\pi^2$. Therefore
	\begin{equation*}
		\abs{B_{(\Sp^3,g_\varepsilon)}(\tilde{x}, R_\varepsilon)} \leq w_3 \varepsilon = \abs{\Sp^3(\sqrt[3]{\varepsilon})} = V^3_{6/\varepsilon^{2/3}}(R_\varepsilon).
	\end{equation*}
	The last equality holds since $R_\varepsilon \geq \pi \sqrt[3]{\varepsilon}$. Therefore, the macroscopic scalar curvature of $(\RP^3,\bar{g}_\varepsilon)$ at a scale $R_\varepsilon$ satisfies $\mscal_{ (\RP^3,\bar{g}_\varepsilon)}(x,R_\varepsilon) \geq 6/\varepsilon^{2/3}$.
\end{proof}

Finally we prove Proposition \ref{pro:RP3} \eqref{ite:UW}.

\begin{proposition} \label{pro:UW}
	Let $\Sigma$ be any closed immersed surface in $(\RP^3,\bar{g}_\varepsilon)$ representing the non-trivial homology class in $H_2(\RP^3;\Z_2) \simeq \Z_2$. Then
	\begin{equation*}
		\UW_1(\Sigma) > \frac{\pi}{4}.
	\end{equation*}
\end{proposition}

The proof of Proposition \ref{pro:UW} is based on the following theorem of Gromov.

\begin{theorem}[{\cite[Proposition F$_1$]{Gromov_1988}}] \label{thm:Gromov_Noncontracting}
	Let $X$ be a metric space. Suppose that $X$ admits a map $\varphi: X \rightarrow \Sp^k(\rho)$ to the $k$-dimensional round sphere of radius $\rho$ which is $L$-Lipschitz and not null-homotopic. Then
	\begin{equation*}
		\UW_{k-1}(X) > \frac{\pi}{2} \cdot \frac{\rho}{L}.
	\end{equation*}
\end{theorem}

\begin{proof}[Proof of Proposition \ref{pro:UW}]
	Suppose that the inclusion map $i: \Sigma \rightarrow \RP^3$ satisfies $i_*[\Sigma] = [\RP^2]$, where $[\Sigma] \in H_2(\Sigma;\Z_2)$ denotes the fundamental class of the surface $\Sigma$ and~$[\RP^2]$ is the generator of~$H_2(\RP^3;\Z_2) \simeq \Z_2$. Consider the map
\begin{equation*}
	\varphi = \bar{\mathcal{H}} \circ i : \Sigma \rightarrow \Sp^2(\tfrac{1}{2}),
\end{equation*}
given by the restriction of the map $\bar{\mathcal{H}}$ to $\Sigma$. The map $\varphi$ is 1-Lipschitz, since it is the restriction of the 1-Lipschitz map $\bar{\mathcal{H}}$ to $\Sigma$.

Let us show that $\varphi$ is not null-homotopic. The Gysin sequence \cite[Section 4.D]{Hatcher_2002} applied to the circle bundle $\bar{\mathcal{H}}: \RP^3 \rightarrow \Sp^2$ yields the exact sequence
\begin{equation*}
	\cdots \rightarrow H^0(\Sp^2;\Z_2) \rightarrow H^2(\Sp^2;\Z_2) \xrightarrow[]{\bar{\mathcal{H}}^*} H^2(\RP^3;\Z_2) \rightarrow H^1(\Sp^2;\Z_2) \rightarrow \cdots.
\end{equation*}
Since $H^1(\Sp^2;\Z_2)$ is trivial, the map
\begin{equation*}
	\bar{\mathcal{H}}^*: H^2(\Sp^2;\Z_2) \simeq \Z_2 \rightarrow H^2(\RP^3;\Z_2) \simeq \Z_2
\end{equation*}
is an epimorphism, and therefore an isomorphism. By the Universal Coefficient Theorem \cite[Theorem 3.2]{Hatcher_2002}, the corresponding induced map in homology
\begin{equation*}
	\bar{\mathcal{H}}_*: H_2(\RP^3;\Z_2) \rightarrow H_2(\Sp^2;\Z_2)
\end{equation*}
is an isomorphism, and it sends the generator $[\RP^2]$ to the fundamental class $[\Sp^2]$. Therefore
\begin{equation*}
	\varphi_* [\Sigma] = \bar{\mathcal{H}}_*[\RP^2] = [\Sp^2],
\end{equation*} 
which implies that $\varphi_* : H_2(\Sigma;\Z_2) \rightarrow H_2(\Sp^2;\Z_2)$ is an isomorphism.

\medskip

Hence the 1-Lipschitz map $\varphi: \Sigma \rightarrow \Sp^2(\frac{1}{2})$ is not null-homotopic. By Theorem \ref{thm:Gromov_Noncontracting}, we conclude that~$\UW_1(\Sigma) > \frac{\pi}{4}$.
\end{proof}

\pagestyle{plain}
\bibliographystyle{alpha}
\bibliography{REF_MacroscopicScalarCurvature}

\end{document}

%% file: Figures/ball.pdf_tex
\begingroup%
  \makeatletter%
  \providecommand\color[2][]{%
    \errmessage{(Inkscape) Color is used for the text in Inkscape, but the package 'color.sty' is not loaded}%
    \renewcommand\color[2][]{}%
  }%
  \providecommand\transparent[1]{%
    \errmessage{(Inkscape) Transparency is used (non-zero) for the text in Inkscape, but the package 'transparent.sty' is not loaded}%
    \renewcommand\transparent[1]{}%
  }%
  \providecommand\rotatebox[2]{#2}%
  \newcommand*\fsize{\dimexpr\f@size pt\relax}%
  \newcommand*\lineheight[1]{\fontsize{\fsize}{#1\fsize}\selectfont}%
  \ifx\svgwidth\undefined%
    \setlength{\unitlength}{334.48818898bp}%
    \ifx\svgscale\undefined%
      \relax%
    \else%
      \setlength{\unitlength}{\unitlength * \real{\svgscale}}%
    \fi%
  \else%
    \setlength{\unitlength}{\svgwidth}%
  \fi%
  \global\let\svgwidth\undefined%
  \global\let\svgscale\undefined%
  \makeatother%
  \begin{picture}(1,0.74576271)%
    \lineheight{1}%
    \setlength\tabcolsep{0pt}%
    \put(0.27920775,0.66076045){\color[rgb]{0,0,0}\makebox(0,0)[lt]{\lineheight{1.25}\smash{\begin{tabular}[t]{l}$L_1$\end{tabular}}}}%
    \put(0.05532689,0.66591536){\color[rgb]{0,0,0}\makebox(0,0)[lt]{\lineheight{1.25}\smash{\begin{tabular}[t]{l}$S_1$\end{tabular}}}}%
    \put(0.1189943,0.25478379){\color[rgb]{0,0,0}\makebox(0,0)[lt]{\lineheight{1.25}\smash{\begin{tabular}[t]{l}$S_3$\end{tabular}}}}%
    \put(0.37391895,0.17866361){\color[rgb]{0,0,0}\makebox(0,0)[lt]{\lineheight{1.25}\smash{\begin{tabular}[t]{l}$S_4$\end{tabular}}}}%
    \put(0.48658138,0.10708334){\color[rgb]{0,0,0}\makebox(0,0)[lt]{\lineheight{1.25}\smash{\begin{tabular}[t]{l}$S_5$\end{tabular}}}}%
    \put(0.61626719,0.00968972){\color[rgb]{0,0,0}\makebox(0,0)[lt]{\lineheight{1.25}\smash{\begin{tabular}[t]{l}$S_6$\end{tabular}}}}%
    \put(0.01565743,0.47073358){\color[rgb]{0,0,0}\makebox(0,0)[lt]{\lineheight{1.25}\smash{\begin{tabular}[t]{l}$S_2$\end{tabular}}}}%
    \put(0.25399677,0.54151136){\color[rgb]{0,0,0}\makebox(0,0)[lt]{\lineheight{1.25}\smash{\begin{tabular}[t]{l}$L_2$\end{tabular}}}}%
    \put(0.40492497,0.47122118){\color[rgb]{0,0,0}\makebox(0,0)[lt]{\lineheight{1.25}\smash{\begin{tabular}[t]{l}$L_3$\end{tabular}}}}%
    \put(0.63628665,0.38822921){\color[rgb]{0,0,0}\makebox(0,0)[lt]{\lineheight{1.25}\smash{\begin{tabular}[t]{l}$L_4$\end{tabular}}}}%
    \put(0.79972406,0.26716216){\color[rgb]{0,0,0}\makebox(0,0)[lt]{\lineheight{1.25}\smash{\begin{tabular}[t]{l}$L_5$\end{tabular}}}}%
    \put(0.71186531,0.08659699){\color[rgb]{0,0,0}\makebox(0,0)[lt]{\lineheight{1.25}\smash{\begin{tabular}[t]{l}$L_6$\end{tabular}}}}%
    \put(0.88409392,0.20709565){\color[rgb]{0.44313725,0.58039216,0.43921569}\makebox(0,0)[lt]{\lineheight{1.25}\smash{\begin{tabular}[t]{l}$D_5$\end{tabular}}}}%
    \put(0.85745065,0.39607355){\color[rgb]{0.44313725,0.58039216,0.43921569}\makebox(0,0)[lt]{\lineheight{1.25}\smash{\begin{tabular}[t]{l}$D_4$\end{tabular}}}}%
    \put(0.71396909,0.52699634){\color[rgb]{0.44313725,0.58039216,0.43921569}\makebox(0,0)[lt]{\lineheight{1.25}\smash{\begin{tabular}[t]{l}$D_3$\end{tabular}}}}%
    \put(0.53674523,0.60270288){\color[rgb]{0.44313725,0.58039216,0.43921569}\makebox(0,0)[lt]{\lineheight{1.25}\smash{\begin{tabular}[t]{l}$D_2$\end{tabular}}}}%
    \put(0.37723076,0.66316277){\color[rgb]{0.44313725,0.58039216,0.43921569}\makebox(0,0)[lt]{\lineheight{1.25}\smash{\begin{tabular}[t]{l}$D_1$\end{tabular}}}}%
    \put(0,0){\includegraphics[width=\unitlength,page=1]{ball.pdf}}%
    \put(0.38562248,0.40761893){\color[rgb]{0,0,0}\makebox(0,0)[lt]{\lineheight{1.25}\smash{\begin{tabular}[t]{l}$x$\end{tabular}}}}%
    \put(0.47290779,0.33718329){\color[rgb]{0.45490196,0.17254902,0.21568627}\makebox(0,0)[lt]{\lineheight{1.25}\smash{\begin{tabular}[t]{l}$B_\Sigma(x,r)$\end{tabular}}}}%
    \put(0,0){\includegraphics[width=\unitlength,page=2]{ball.pdf}}%
  \end{picture}%
\endgroup%

%% file: Figures/deformation.pdf_tex
\begingroup%
  \makeatletter%
  \providecommand\color[2][]{%
    \errmessage{(Inkscape) Color is used for the text in Inkscape, but the package 'color.sty' is not loaded}%
    \renewcommand\color[2][]{}%
  }%
  \providecommand\transparent[1]{%
    \errmessage{(Inkscape) Transparency is used (non-zero) for the text in Inkscape, but the package 'transparent.sty' is not loaded}%
    \renewcommand\transparent[1]{}%
  }%
  \providecommand\rotatebox[2]{#2}%
  \newcommand*\fsize{\dimexpr\f@size pt\relax}%
  \newcommand*\lineheight[1]{\fontsize{\fsize}{#1\fsize}\selectfont}%
  \ifx\svgwidth\undefined%
    \setlength{\unitlength}{549.92125984bp}%
    \ifx\svgscale\undefined%
      \relax%
    \else%
      \setlength{\unitlength}{\unitlength * \real{\svgscale}}%
    \fi%
  \else%
    \setlength{\unitlength}{\svgwidth}%
  \fi%
  \global\let\svgwidth\undefined%
  \global\let\svgscale\undefined%
  \makeatother%
  \begin{picture}(1,0.38659794)%
    \lineheight{1}%
    \setlength\tabcolsep{0pt}%
    \put(0,0){\includegraphics[width=\unitlength,page=1]{deformation.pdf}}%
    \put(0.26768688,0.07772151){\color[rgb]{0.44313725,0.58039216,0.43921569}\makebox(0,0)[lt]{\lineheight{1.25}\smash{\begin{tabular}[t]{l}$D_5$\end{tabular}}}}%
    \put(0.20642754,0.13902924){\color[rgb]{0.44313725,0.58039216,0.43921569}\makebox(0,0)[lt]{\lineheight{1.25}\smash{\begin{tabular}[t]{l}$D_4$\end{tabular}}}}%
    \put(0.17199672,0.20080214){\color[rgb]{0.44313725,0.58039216,0.43921569}\makebox(0,0)[lt]{\lineheight{1.25}\smash{\begin{tabular}[t]{l}$D_3$\end{tabular}}}}%
    \put(0.13694533,0.26431821){\color[rgb]{0.44313725,0.58039216,0.43921569}\makebox(0,0)[lt]{\lineheight{1.25}\smash{\begin{tabular}[t]{l}$D_2$\end{tabular}}}}%
    \put(0.13051114,0.32460184){\color[rgb]{0.44313725,0.58039216,0.43921569}\makebox(0,0)[lt]{\lineheight{1.25}\smash{\begin{tabular}[t]{l}$D_1$\end{tabular}}}}%
    \put(0.01207095,0.07623653){\color[rgb]{0.44313725,0.58039216,0.43921569}\makebox(0,0)[lt]{\lineheight{1.25}\smash{\begin{tabular}[t]{l}$\Sigma$\end{tabular}}}}%
    \put(0,0){\includegraphics[width=\unitlength,page=2]{deformation.pdf}}%
    \put(0.87912176,0.02332487){\color[rgb]{0.44313725,0.58039216,0.43921569}\makebox(0,0)[lt]{\lineheight{1.25}\smash{\begin{tabular}[t]{l}$D'_5$\end{tabular}}}}%
    \put(0.85990473,0.15925289){\color[rgb]{0.44313725,0.58039216,0.43921569}\makebox(0,0)[lt]{\lineheight{1.25}\smash{\begin{tabular}[t]{l}$D'_4$\end{tabular}}}}%
    \put(0.7349542,0.22425447){\color[rgb]{0.44313725,0.58039216,0.43921569}\makebox(0,0)[lt]{\lineheight{1.25}\smash{\begin{tabular}[t]{l}$D'_3$\end{tabular}}}}%
    \put(0.70720484,0.2846206){\color[rgb]{0.44313725,0.58039216,0.43921569}\makebox(0,0)[lt]{\lineheight{1.25}\smash{\begin{tabular}[t]{l}$D'_2$\end{tabular}}}}%
    \put(0.68656119,0.32516443){\color[rgb]{0.44313725,0.58039216,0.43921569}\makebox(0,0)[lt]{\lineheight{1.25}\smash{\begin{tabular}[t]{l}$D'_1$\end{tabular}}}}%
    \put(0.57058751,0.0762369){\color[rgb]{0.44313725,0.58039216,0.43921569}\makebox(0,0)[lt]{\lineheight{1.25}\smash{\begin{tabular}[t]{l}$\Sigma'$\end{tabular}}}}%
    \put(0,0){\includegraphics[width=\unitlength,page=3]{deformation.pdf}}%
  \end{picture}%
\endgroup%

%% file: Figures/prolate.pdf_tex
\begingroup%
  \makeatletter%
  \providecommand\color[2][]{%
    \errmessage{(Inkscape) Color is used for the text in Inkscape, but the package 'color.sty' is not loaded}%
    \renewcommand\color[2][]{}%
  }%
  \providecommand\transparent[1]{%
    \errmessage{(Inkscape) Transparency is used (non-zero) for the text in Inkscape, but the package 'transparent.sty' is not loaded}%
    \renewcommand\transparent[1]{}%
  }%
  \providecommand\rotatebox[2]{#2}%
  \newcommand*\fsize{\dimexpr\f@size pt\relax}%
  \newcommand*\lineheight[1]{\fontsize{\fsize}{#1\fsize}\selectfont}%
  \ifx\svgwidth\undefined%
    \setlength{\unitlength}{155.90551181bp}%
    \ifx\svgscale\undefined%
      \relax%
    \else%
      \setlength{\unitlength}{\unitlength * \real{\svgscale}}%
    \fi%
  \else%
    \setlength{\unitlength}{\svgwidth}%
  \fi%
  \global\let\svgwidth\undefined%
  \global\let\svgscale\undefined%
  \makeatother%
  \begin{picture}(1,1.21818182)%
    \lineheight{1}%
    \setlength\tabcolsep{0pt}%
    \put(0,0){\includegraphics[width=\unitlength,page=1]{prolate.pdf}}%
    \put(0.52151507,0.01562084){\color[rgb]{0,0,0}\makebox(0,0)[lt]{\lineheight{1.25}\smash{\begin{tabular}[t]{l}$x_0$\end{tabular}}}}%
    \put(0.56977872,0.89918331){\color[rgb]{0,0,0}\makebox(0,0)[lt]{\lineheight{1.25}\smash{\begin{tabular}[t]{l}$x$\end{tabular}}}}%
    \put(0.53796903,1.04944177){\color[rgb]{0,0,0}\makebox(0,0)[lt]{\lineheight{1.25}\smash{\begin{tabular}[t]{l}$h_0+R$\end{tabular}}}}%
    \put(0.56371436,0.76598795){\color[rgb]{0,0,0}\makebox(0,0)[lt]{\lineheight{1.25}\smash{\begin{tabular}[t]{l}$h_0-R$\end{tabular}}}}%
    \put(0.57599404,0.59484177){\color[rgb]{0,0,0}\makebox(0,0)[lt]{\lineheight{1.25}\smash{\begin{tabular}[t]{l}$\varepsilon$\end{tabular}}}}%
    \put(0.00466661,0.88228406){\color[rgb]{0.44313725,0.58039216,0.43921569}\makebox(0,0)[lt]{\lineheight{1.25}\smash{\begin{tabular}[t]{l}$B_{E^k(\varepsilon, a)}(x,R)$\end{tabular}}}}%
    \put(0,0){\includegraphics[width=\unitlength,page=2]{prolate.pdf}}%
  \end{picture}%
\endgroup%

%% file: MacroscopicScalarCurvature.bbl
\begin{thebibliography}{BBEN10}

\bibitem[ABG24]{Alpert_Balitskiy_Guth_2024}
Hannah Alpert, Alexey Balitskiy, and Larry Guth.
\newblock Macroscopic scalar curvature and codimension 2 width.
\newblock {\em J. Topol. Anal.}, 16(6):979--987, 2024.

\bibitem[AF17]{Alpert_Kei_2017}
Hannah Alpert and Kei Funano.
\newblock Macroscopic scalar curvature and areas of cycles.
\newblock {\em Geom. Funct. Anal.}, 27(4):727--743, 2017.

\bibitem[Alp22]{Alpert_2022}
Hannah Alpert.
\newblock Macroscopic stability and simplicial norms of hypersurfaces.
\newblock {\em Comm. Anal. Geom.}, 30(5):949--959, 2022.

\bibitem[BBEN10]{Bray_Brendle_Eichmair_Neves_2010}
Hubert Bray, Simon Brendle, Michael Eichmair, and Andr\'e Neves.
\newblock Area-minimizing projective planes in 3-manifolds.
\newblock {\em Comm. Pure Appl. Math.}, 63(9):1237--1247, 2010.

\bibitem[BBN10]{Bray_Brendle_Neves_2010}
Hubert Bray, Simon Brendle, and Andr\'e Neves.
\newblock Rigidity of area-minimizing two-spheres in three-manifolds.
\newblock {\em Comm. Anal. Geom.}, 18(4):821--830, 2010.

\bibitem[BZ88]{1988_Burago_Zalgaller}
Yu.~D. Burago and V.~A. Zalgaller.
\newblock {\em Geometric inequalities}, volume 285 of {\em Grundlehren der
  mathematischen Wissenschaften [Fundamental Principles of Mathematical
  Sciences]}.
\newblock Springer-Verlag, Berlin, 1988.
\newblock Translated from the Russian by A. B. Sosinski\u i, Springer Series in
  Soviet Mathematics.

\bibitem[CLZ24]{Chu_Lee_Zhu_2024}
Jianchun Chu, Man-Chun Lee, and Jintian Zhu.
\newblock Homological $n$-systole in $(n+1)$-manifolds and bi-{R}icci
  curvature.
\newblock arXiv:2410.20785, 2024.

\bibitem[GHL04]{GHL_2004}
Sylvestre Gallot, Dominique Hulin, and Jacques Lafontaine.
\newblock {\em Riemannian geometry}.
\newblock Universitext. Springer-Verlag, Berlin, third edition, 2004.

\bibitem[Gro83]{Gromov_1983}
Mikhael Gromov.
\newblock Filling {R}iemannian manifolds.
\newblock {\em J. Differential Geom.}, 18(1):1--147, 1983.

\bibitem[Gro86]{Gromov_1986}
Mikhael Gromov.
\newblock Large {R}iemannian manifolds.
\newblock In {\em Curvature and topology of {R}iemannian manifolds ({K}atata,
  1985)}, volume 1201 of {\em Lecture Notes in Math.}, pages 108--121.
  Springer, Berlin, 1986.

\bibitem[Gro88]{Gromov_1988}
Mikhael Gromov.
\newblock Width and related invariants of {R}iemannian manifolds.
\newblock Number 163-164, pages 6, 93--109, 282. 1988.
\newblock On the geometry of differentiable manifolds (Rome, 1986).

\bibitem[Gro07]{Gromov_1981}
Mikhael Gromov.
\newblock {\em Metric structures for {R}iemannian and non-{R}iemannian spaces}.
\newblock Modern Birkh\"auser Classics. Birkh\"auser Boston, Inc., Boston, MA,
  english edition, 2007.
\newblock Based on the 1981 French original, With appendices by M. Katz, P.
  Pansu and S. Semmes, Translated from the French by Sean Michael Bates.

\bibitem[Gut10a]{Guth_2010_Metaphors}
Larry Guth.
\newblock Metaphors in systolic geometry.
\newblock In {\em Proceedings of the {I}nternational {C}ongress of
  {M}athematicians. {V}olume {II}}, pages 745--768. Hindustan Book Agency, New
  Delhi, 2010.

\bibitem[Gut10b]{Guth_2010_Systolic}
Larry Guth.
\newblock Systolic inequalities and minimal hypersurfaces.
\newblock {\em Geom. Funct. Anal.}, 19(6):1688--1692, 2010.

\bibitem[Gut17]{Guth_2017}
Larry Guth.
\newblock Volumes of balls in {R}iemannian manifolds and {U}ryson width.
\newblock {\em J. Topol. Anal.}, 9(2):195--219, 2017.

\bibitem[Hat02]{Hatcher_2002}
Allen Hatcher.
\newblock {\em Algebraic topology}.
\newblock Cambridge University Press, Cambridge, 2002.

\bibitem[LLNR22]{LLNR_2022}
Yevgeny Liokumovich, Boris Lishak, Alexander Nabutovsky, and Regina Rotman.
\newblock Filling metric spaces.
\newblock {\em Duke Math. J.}, 171(3):595--632, 2022.

\bibitem[Nab22]{Nabutovsky_2022}
Alexander Nabutovsky.
\newblock Linear bounds for constants in {G}romov's systolic inequality and
  related results.
\newblock {\em Geom. Topol.}, 26(7):3123--3142, 2022.

\bibitem[OMS09]{OMS_2009}
Keith Oldham, Jan Myland, and Jerome Spanier.
\newblock {\em An atlas of functions}.
\newblock Springer, New York, second edition, 2009.
\newblock With Equator, the atlas function calculator, With 1 CD-ROM (Windows).

\bibitem[Pap20]{Papasoglu_2020}
Panos Papasoglu.
\newblock Uryson width and volume.
\newblock {\em Geom. Funct. Anal.}, 30(2):574--587, 2020.

\bibitem[Ric20]{Richard_2020}
Thomas Richard.
\newblock On the 2-systole of stretched enough positive scalar curvature
  metrics on {$\Bbb S^2\times\Bbb S^2$}.
\newblock {\em SIGMA Symmetry Integrability Geom. Methods Appl.}, 16:Paper No.
  136, 7, 2020.

\bibitem[Sab22]{Sabourau_2022}
St\'ephane Sabourau.
\newblock Macroscopic scalar curvature and local collapsing.
\newblock {\em Ann. Sci. \'Ec. Norm. Sup\'er. (4)}, 55(4):919--936, 2022.

\bibitem[Ste22]{Stern_2022}
Daniel~L. Stern.
\newblock Scalar curvature and harmonic maps to {$S^1$}.
\newblock {\em J. Differential Geom.}, 122(2):259--269, 2022.

\bibitem[Zhu20]{Zhu_2020}
Jintian Zhu.
\newblock Rigidity of area-minimizing {$2$}-spheres in {$n$}-manifolds with
  positive scalar curvature.
\newblock {\em Proc. Amer. Math. Soc.}, 148(8):3479--3489, 2020.

\bibitem[Zhu23]{Zhu_2023}
Jintian Zhu.
\newblock Rigidity results for complete manifolds with nonnegative scalar
  curvature.
\newblock {\em J. Differential Geom.}, 125(3):623--644, 2023.

\end{thebibliography}
